\documentclass{amsart}

\usepackage{color,graphicx,amssymb,latexsym,amsfonts,txfonts,amsmath,amsthm}
\usepackage{pdfsync,enumitem}
\usepackage{amsmath,amscd}
\usepackage[all,cmtip]{xy}

\usepackage{hyperref}
\hypersetup{
    colorlinks=true,       
    linkcolor=blue,          
    citecolor=blue,        
    filecolor=blue,      
    urlcolor=blue           
}

\input epsf

\newtheorem{theo}{Theorem}[section]
\newtheorem{coro}{Corollary}[section]

\newtheorem{claim}{Claim}

\theoremstyle{remark}  
\newtheorem{rema}{\bf Remark}[section]

\begin{document}

\title{On $p$-gonal fields of definition}
\author{Ruben A. Hidalgo}

\thanks{Partially supported by projects Fondecyt 1190001 and 1220261}

\address{Departamento de Matem\'atica y Estad\'{\i}stica, Universidad de La Frontera. Temuco, Chile}
\email{ruben.hidalgo@ufrontera.cl}

\begin{abstract} 
Let $S$ be a closed Riemann surface of genus $g \geq 2$ and $\varphi$ be a conformal automorphism of $S$ of prime order $p$ such that $S/\langle \varphi \rangle$ has genus zero. 
Let ${\mathbb K} \leq {\mathbb C}$ be a field of definition of $S$.
We provide an argument for the existence of a field extension ${\mathbb F}$ of ${\mathbb K}$, of degree at most $2(p-1)$, for which $S$ is definable by a curve of the form $y^{p}=F(x) \in {\mathbb F}[x]$, in which case $\varphi$ corresponds to $(x,y) \mapsto (x,e^{2 \pi i/p} y)$. 
If, moreover, $\varphi$ is also definable over ${\mathbb K}$, then ${\mathbb F}$ can be chosen to be at most a quadratic extension of ${\mathbb K}$. 
For $p=2$, that is when $S$ is hyperelliptic and $\varphi$ is its hyperelliptic involution, this fact is  due to Mestre (for even genus) and Huggins and Lercier-Ritzenthaler-Sijslingit in the case that ${\rm Aut}(S)/\varphi\rangle$ is non-trivial. 
\end{abstract}

\maketitle

\section{Introduction}
In \cite{Schwarz}, H. A. Schwarz proved that the group ${\rm Aut}(S)$ of conformal automorphisms of a closed Riemann surface 
$S$ of genus $g \geq 2$ is finite. Later, in \cite{Hurwitz}, A. Hurwitz obtained the upper bound  $|{\rm Aut}(S)| \leq 84(g-1)$ (this is known as the Hurwitz's bound).  

Let $p \geq 2$ be a prime integer. We say that a closed Riemann surface $S$ is {\it cyclic $p$-gonal} if there exists $\varphi \in {\rm Aut}(S)$ of order $p$ such that the quotient orbifold $S/\langle \varphi \rangle$ has genus zero. In this case, $\varphi$ is called a {\it $p$-gonal automorphism} and the cyclic group $\langle \varphi \rangle$ a {\it $p$-gonal group} of $S$. The case $p=2$ corresponds to $S$ being hyperelliptic and $\varphi$ its (unique) hyperelliptic involution. 
The case $p=3$  was studied by R.D.M. Accola in \cite{Accola}. In \cite{Gabino} G. Gonz\'alez-Diez proved that $p$-gonal groups are unique up to conjugation in  ${\rm Aut}(S)$ (if $p \geq 5n-7$, where $n \geq 3$ is the number of fixed points of $\varphi$, then $\langle \varphi \rangle$ is the unique $p$-group in ${\rm Aut}(S)$ \cite{H:pgrupo}). Results concerning automorphisms of $p$-gonal Riemann surfaces can be found, for instance, in \cite{BCI2, BW, BCG,  CI3, GWW, W2}.

As a consequence of the Riemann-Roch theorem, a closed Riemann surface $S$ can be described by an (either affine or projective) irreducible complex algebraic curve (i.e., after normalization of the curve if necessary, it carries a Riemann surface structure which is biholomorphic to that of $S$). For instance, Equation \eqref{eq1} (respectively, Equation \eqref{eq2}) given below provides an affine (respectively, a projective) irreducible algebraic curve $E$ (respectively, $\widehat{E}$) representing a cyclic $p$-gonal surface (the affine model is smooth, but the projective one has a singularity at the point $[0:1:0]$).

Let ${\mathbb K}$ be a subfield of the field ${\mathbb C}$ of complex numbers. 
If one may find an irreducible algebraic curve representing $S$, which is defined as the common zeroes of some polynomials with coefficients in ${\mathbb K}$, then we say that $S$ is {\it definable} over ${\mathbb K}$ (and that ${\mathbb K}$ is a {\it field of definition} of $S$). The intersection of all the fields of definition of $S$, called the {\it field of moduli} of $S$ (see Section \ref{Sec:FOM}), is not in general a field of definition.

If we are given $G< {\rm Aut}(S)$ and the geometrical structure of the quotient orbifold $S/G$, then it is not a simple task to find an algebraic curve for $S$ reflecting the action of $G$. A family of surfaces for which this is well known is the case of cyclic $p$-gonal Riemann surfaces as we proceed to recall below.

Let $S$ be a $p$-gonal Riemann surface, $\varphi \in {\rm Aut}(S)$ be a $p$-gonal automorphism and $\pi:S \to \widehat{\mathbb C}$ be 
a regular branched cover with $\langle \varphi \rangle$ its deck group. Let $\{a_{1},\ldots, a_{m}\} \subset \widehat{\mathbb C}$ be the set of branch values of $\pi$. If  $a_{j} \neq \infty$, for every $j=1,\ldots, m$, then 
there exist integers $n_{1},...,n_{m} \in \{1,...,p-1\}$, $n_{1}+\cdots+n_{m} \equiv 0 \; {\rm mod \;} p$, such that $S$ is defined by the affine, irreducible and smooth {\it $p$-gonal curve} with equation 
\begin{equation}\label{eq1}
E: \; y^{p}=F(x)=\prod_{j=1}^{m} (x-a_{j})^{n_{j}} \in {\mathbb C}[x].
\end{equation} 

If one of the branch values is equal to $\infty$, say $a_{m}=\infty$, then in \eqref{eq1} we delete the corresponding factor $(x-a_{m})^{n_{m}}$ and assume $n_{1}+\cdots+n_{m-1} \nequiv 0 \; {\rm mod \;} p$.  In the hyperelliptic case, i.e., $p=2$, in the above one has $m \in \{2g+1,2g+2\}$ and $n_{j}=1$.
In the above affine algebraic model, $\pi(x,y)=x$ and $\varphi(x,y)=(x,\omega_{p} y)$, where $\omega_{p}=e^{2 \pi i/p}$.

\begin{rema}
An irreducible projective algebraic curve is obtained from the above affine one as
\begin{equation}\label{eq2}
\widehat{E}: \; y^{p}z^{n_{1}+\cdots+n_{m}-p}=\prod_{j=1}^{m} (x-a_{j}z)^{n_{j}}.
\end{equation} 

Note that the projective curve $\widehat{E}$ is not smooth at the point $[0:1:0]$. After normalization of the curve, one obtains a closed Riemann surface which is biholomorphic to $S$. In this case, $\pi([x:y:z])=x/z$ and 
$\varphi([x:y:z])=[x:\omega_{p} y:z]$.
\end{rema}

If ${\mathbb F}$ is a subfield of ${\mathbb C}$ such that in \eqref{eq1} we have $F(x) \in {\mathbb F}[x]$,  then we say that ${\mathbb F}$ is a {\it $p$-gonal field of definition} of $S$ (and that $S$ is {\it cyclically $p$-gonally defined over ${\mathbb F}$}). 
Note that there are infinitely many different $p$-gonal fields of definition for $S$ (a, for instance, if $T$ is a M\"obius transformation, then we may replace the values $a_{j}$ by $T(a_{j})$).

Given a field of definition of a $p$-gonal Riemann surface $S$, it is not clear at a first sight if it is a $p$-gonal field of definition. 
Also, it might be that a minimal $p$-gonal field of definition is not a minimal field of definition (see the exceptional case $(m,p)=(4,3)$ in Section \ref{auto}).  The aim of this paper is to provide an argument to show that, given any field of definition ${\mathbb K}$ of $S$, there is a $p$-gonal field of definition ${\mathbb F}$ which is an extension of degree at most $2(p-1)$ over ${\mathbb K}$.

If $\varphi$ is an automorphism of $S$, then we say that $S$ and $\varphi$ are simultaneously defined over ${\mathbb K}$ if there is an algebraic curve model of $S$, defined over ${\mathbb K}$, such that $\varphi$ is given by a rational map on it with coefficients in ${\mathbb K}$.

\begin{theo}\label{teo1} Let $S$ be a cyclic $p$-gonal Riemann surface of genus $g \geq 2$, with a $p$-gonal automorphism $\varphi$, and let ${\mathbb K}$ be a field of definition of $S$. Then 
\begin{enumerate}[leftmargin=*,align=left]
\item There is $p$-gonal field of definition of $S$, this being an extension of degree at most $2(p-1)$ of ${\mathbb K}$ (which is also a field of definition of $\varphi$).
\item If both $S$ and $\varphi$ are simultaneously defined over ${\mathbb K}$, then there is a $p$-gonal field of definition of $S$, this being an extension of degree at most two of ${\mathbb K}$.

\item If in equation \eqref{eq1} $n_{1}=\cdots=n_{m}$, then there is a $p$-gonal field of definition of $S$, this being an extension of degree at most two of ${\mathbb K}$.
\end{enumerate}
\end{theo}

\begin{rema}
Theorem \ref{teo1} still valid if we change ${\mathbb C}$ to any algebraically closed field, where in positive characteristic we need to assume that $p$ is different from the characteristic.
\end{rema}

\begin{rema}
For each integer $n \geq 2$, not necessarily prime, the definition of cyclic $n$-gonal Riemann surface $S$, $n$-gonal automorphism $\varphi$ and $n$-gonal group $\langle \varphi \rangle$ is the same as for the prime situation. In the particular case that every fixed point of a non-trivial power $\varphi^{k}$ is also a fixed point of $\varphi$, the 
definition of an $n$-gonal curve is the same as in \eqref{eq1}, but replacing $p$ by $n$ and assuming each the exponent $n_{j}$ to be relatively prime to $n$. In this case, 
under the assumption that $S$ has a unique $n$-gonal group $\langle \varphi \rangle$ (this is the situation for generalized superelliptic Riemann surfaces \cite{HQS}), then the arguments of the proof of Theorem \ref{teo1} allows us to obtain that: if ${\mathbb K}$ is a field of definition of $S$, then there is an $n$-gonal field of definition of $S$, this being an extension of degree at most $2\phi(n)$ of ${\mathbb K}$, where $\phi(n)$ is the $\phi$-Euler function.
\end{rema}

 \subsection{An application to hyperelliptic surfaces}
 Let $S$ be a hyperelliptic Riemann surface (i.e., $p=2$) with hyperelliptic involution $\varphi$ and let ${\mathbb K}$ be a field of definition of $S$. As $\varphi$ is unique, one may consider the group ${\rm Aut}_{red}(S):={\rm Aut}(S)/\langle \varphi \rangle$, called the reduced group of automorphisms of $S$.

For even genus, in \cite{Mestre},  J-F. Mestre proved that $S$ is also hyperelliptically definable over ${\mathbb K}$. If the genus is odd, then the previous fact is in general false; as can be seen from examples in \cite{Fuertes,FG,LR,LRS}. In \cite{Huggins}, B. Huggings  proved that if ${\rm Aut}_{red}(S)$ is neither trivial or cyclic, then $S$ is also hyperelliptically definable over ${\mathbb K}$. In \cite{LRS},
R. Lercier, C. Ritzenthaler and J. Sijslingit proved that $S$ can be hyperelliptically defined over a quadratic extension of ${\mathbb K}$ if the reduced group is a non-trivial cyclic group. Our theorem asserts that this fact still valid even if the reduced group is trivial.

\begin{coro}\label{coro1}
If ${\mathbb K}$ is a field of definition of a hyperelliptic Riemann surface, then it is  hyperelliptically definable over an extension of degree at most two of ${\mathbb K}$.
\end{coro}

\section{An application to fields of moduli}\label{Sec:FOM}
Let $S$ be a closed Riemann surface and let $C$ be an irreducible algebraic curve representing it. The {\it field of moduli} ${\mathcal M}_{S}$ of $S$ is the fixed field of the group $\Gamma_{C}=\{\sigma \in {\rm Aut}({\mathbb C}/{\mathbb Q}): C^{\sigma} \cong C\}$; this field does not depend on the choice of the algebraic model $C$. In \cite{Koizumi}, S. Koizumi proved that ${\mathcal M}_{S}$ coincides with the intersection of all fields of definition of $S$, but in general it might not be a field of definition \cite{Earle1, Earle2, Hid, Huggins, Kontogeorgis}. If ${\rm Aut}(S)$ is trivial (the generic situation for $g \geq 3$), then Weil's descent theorem \cite{Weil} asserts that ${\mathcal M}_{S}$ is a field of definition of $S$.  In \cite{W}, J. Wolfart proved that if $S/{\rm Aut}(S)$ is the Riemann sphere with exactly $3$ cone points (i.e., $S$ is quasiplatonic), then ${\mathcal M}_{S}$ is also a field of definition of $S$. In a more general setting, if $S/{\rm Aut}(S)$ has genus zero, then it is known that $S$ is definable over an extension of degree at most two of ${\mathcal M}_{S}$ (see \cite{H:FOD/FOM} for a more general statement).

Now, let $S$ be a $p$-gonal Riemann surface of genus $g \geq 2$ and let $G=\langle \varphi \rangle<{\rm Aut}(S)$ be a $p$-gonal group. 
As previously noted, $S$ is either definable over ${\mathcal M}_{S}$ or over a suitable quadratic extension of it (but it might not be cyclically $p$-gonally definable over such a minimal field of definition). In the case that $G$ is not a unique $p$-gonal subgroup, in \cite{W1}, A. Wootton noted that 
$S$ can be cyclically $p$-gonally defined over an extension of degree at most $2$ of its field of moduli. In the case that $G$ is the unique $p$-gonal 
subgroup, the quotient group ${\rm Aut}(S)/G$ is called the {\it reduced group} of $S$. In \cite{Kontogeorgis}, A. Kontogeorgis  proved that if the reduced group is neither trivial or a cyclic group, then $S$ can always be defined over its field of moduli. So, a direct consequence of Theorem \ref{teo1} is the following.

\begin{coro}\label{coro:moduli}
Let $S$ be a cyclic $p$-gonal Riemann surface with a $p$-gonal group $G=\langle \varphi \rangle$.
\begin{enumerate}[leftmargin=*,align=left]
\item If $G$ is  not a normal subgroup of ${\rm Aut}(S)$, then $S$ is cyclically $p$-gonally definable over an extension of degree at most two of ${\mathcal M}_{S}$. 

\item If $G$ is a normal subgroup of ${\rm Aut}(S)$ and ${\rm Aut}(S)/G$ is different from the trivial group or a cyclic group, then $S$ is cyclically $p$-gonally definable over an extension of degree at most $2(p-1)$ of ${\mathcal M}_{S}$. Moreover, if $\varphi$ also is defined over ${\mathcal M}_{S}$, then the extension can be chosen to be of degree at most two.

\item If $G={\rm Aut}(S)$, then $S$ is cyclically $p$-gonally definable over an extension of degree at most $4(p-1)$ of its field of moduli. Moreover, if $\varphi$ also is defined over ${\mathcal M}_{S}$, then the extension can be chosen of degree at most $4$.
\end{enumerate}
\end{coro}

As every hyperelliptic Riemann surface is definable over an extension of degree at most two of its field of moduli, 
Corollary \ref{coro1} asserts the following.

\begin{coro}
Every hyperelliptic Riemann surface is hyperelliptically definable over an extension of degree at most $4$ of its field of moduli. Moreover, if either (i) the genus is even or (ii) the genus is odd and the reduced group is not trivial, then the hyperelliptic Riemann surface is hyperelliptically defined over an extension of degree at most $2$ of its field of moduli. 
\end{coro}

Examples of hyperelliptic Riemann surface with trivial reduced group which cannot be defined over their field of moduli were provided by C. J. Earle \cite{Earle1,Earle2} and G. Shimura \cite{Shimura}. The same type of examples, but with non-trivial cyclic reduced group, were provided by B. Huggins \cite{Huggins}.

\section{Proof of Theorem \ref{teo1}}
We assume the $p$-gonal Riemann surface $S$ to be provided by an irreducible curve $C$, defined over 
a subfield ${\mathbb K}$ of ${\mathbb C}$.
 If $\overline{{\mathbb K}}$ is the algebraic closure of ${\mathbb K}$ inside ${\mathbb C}$, then (in this algebraic model)
the $p$-gonal automorphism $\varphi$ is given by a rational map defined over $\overline{\mathbb K}$. We divide the arguments depending on the uniqueness of the cyclic group $\langle \varphi \rangle$.

\subsection{The case when  $\langle \varphi \rangle$ is not unique}\label{auto}
The following result, due to A. Wootton, describes those cases were the uniqueness fails.

\begin{theo}[A. Wootton \cite{W1}]\label{Wootton1}
Let $S$ be a cyclic $p$-gonal Riemann surface of genus $g \geq 2$ and let $m=2(g+p-1)/(p-1)$. If $(m,p)$ is different from any the following tuples
$$(i) \; (3,7), \; (ii) \; (4,3), \; (iii)\; (4,5), \;(iv) \; (5,3), \; (v) (p,p),\; p \geq 5, \; (vi) \; (2p,p),\; p \geq 3,$$
then $S$ has a unique $p$-gonal group. 
\end{theo}

In the same paper, Wootton describes the exceptional cyclic $p$-gonal Riemann surfaces, ie., where the $p$-gonal group is non-unique.

\begin{enumerate}[leftmargin=*,align=left]
\item Case $(m,p)=(3,7)$ corresponds to Klein's quartic (a non-hyperelliptic Riemann surface of genus $3$) $x^{3}y+y^{3}z+z^{3}x=0$, whose group of automorphisms is ${\rm PGL}_{2}(7)$ (of order $168$). This surface is cyclically $7$-gonally defined as $y^{7}=x^{2}(x-z)z^{4}$.

\item Case $(m,p)=(4,3)$ corresponds to the genus $2$ Riemann surface defined hyperelliptically by $y^{2}z^{3}=x(x^{4}-z^{4})$, whose group of automorphisms is ${\rm GL}_{2}(3)$ (of order $48$). This surface is cyclically $3$-gonally defined as $y^{3}z^{3}=(x^{2}-z^{2})(x^{2}-(15\sqrt{3}-26)z^{2})^2$.

\item Case $(m,p)=(4,5)$ corresponds to the genus $4$ non-hyperelliptic Riemann surface, called Bring's curve, which is the complete intersection of the quadric $x_{1}x_{4}+x_{2}x_{3}=0$ and the cubic $x_{1}^{2}x_{3}+x_{2}^{2}x_{1}+x_{3}^{2}x_{4}+x_{4}^{2}x_{2}=0$ in the $3$-dimensional complex projective space. Its group of automorphisms is  ${\mathfrak S}_{5}$, the symmetric group in five letters ${\mathfrak S}_{5}$. This surface is cyclically $5$-gonally defined as $y^{5}z^{5}=(x^{2}-z^{2})(x^{2}+z^{2})^{4}$.

\item Case $(m,p)=(5,3)$ corresponds to the genus $3$ non-hyperelliptic closed Riemann surface $x^{4}+y^{4}+z^{4}+2i\sqrt{3}z^{2}y^{2}=0$, whose group of automorphisms has order $48$. The quotient of that surface by its group of automorphisms has signature $(0;2,3,12)$. This surface is cyclically $3$-gonally defined as $y^{3}z^{3}=x^{2}(x^{4}-z^{4})$.

\item Case $(m,p)=(p,p)$, where $p \geq 5$, corresponds to the Fermat curve $x^{p}+y^{p}+z^{p}=0$, whose group of automorphisms is ${\mathbb Z}_{p}^{2} \rtimes {\mathfrak S}_{3}$. This is already in a $p$-gonal form as $y^{p}=-z^{p}-x^{p}$.

\item Case $(m,p)=(2p,p)$, where $p \geq 3$. There is a $1$-dimensional family with group of automorphisms ${\mathbb Z}_{p}^{2} \rtimes {\mathbb Z}_{2}^{2}$ (the quotient by that group has signature $(0;2,2,2,p)$). Also, there is a surface with group of automorphisms ${\mathbb Z}_{p}^{2} \rtimes D_{4}$ (the quotient by that group has signature $(0;2,4,2p)$.
These surfaces are cyclically $p$-gonally defined as $y^{p}z^{p}=(x^{p}-a^{p}z^{p})(x^{p}-z^{p}/a^{p})=x^{2p}-(a^{p}+1/a^{p})x^{p}z^{p}+z^{2p}$.
\end{enumerate}

\medskip

Note that, in all the above exceptional cases, the surface $S$ is cyclically $p$-gonally defined over an extension of degree at most $2$ over the field of moduli. In fact, with only the exception of case (2), $S$ is cyclically $p$-gonally defined over its field of moduli. So, we are done in this situation.

\subsection{The case when  $\langle \varphi \rangle$ is unique}
We now assume that  $\langle \varphi \rangle$ is unique. Set $\Gamma={\rm Gal}(\overline{{\mathbb K}}/{\mathbb K})$. 
Let us consider a rational map $\pi:C \to {\mathbb P}_{\overline{{\mathbb K}}}^{1}$, defined over $\overline{\mathbb K}$, which is 
a regular branched covering with $\langle \varphi \rangle$ as its deck group and whose branch values are $a_{1},..., a_{m} \in {\mathbb C}$ (in fact, these values belong to $\overline{\mathbb K}$).
Let the integers $n_{1},...,n_{m} \in \{1,...,p-1\}$, $n_{1}+\cdots+n_{m} \equiv 0 \; {\rm mod \;} p$, such that $C$ is isomorphic to a $p$-gonal curve $E$ with equation \eqref{eq1}.

\subsubsection{Proof of Part(1)}
Let us recall that $\varphi$ is already defined over $\overline{\mathbb K}$. We start by observing that in fact it is defined over an extension of ${\mathbb K}$ of degree at most $p-1$.

\begin{claim}
The rational map $\varphi$ is defined over an extension ${\mathbb K}_{1}$ of ${\mathbb K}$ of degree at most $p-1$.
\end{claim}
\begin{proof}
If $\sigma \in \Gamma$, then $\varphi^{\sigma}$ is an automorphism of order $p$ of $C^{\sigma}=C$. As we are assuming the uniqueness of $\langle \varphi \rangle$, we must have that $\varphi^{\sigma} \in \Omega:=\{\varphi, \varphi^{2}, \ldots, \varphi^{p-1}\}$. In particular, the subgroup $A$ of $\Gamma$ consisting of those $\sigma$ such that $\varphi^{\sigma}=\varphi$ must have index at most the cardinality of the set $\Omega$, which is $p-1$. This asserts that $\varphi$ is defined over the fixed field ${\mathbb K}_{1}$ of $A$, which is an extension of degree at most $p-1$ of 
${\mathbb K}$.
\end{proof}

Set $\Gamma_{1}={\rm Gal}(\overline{{\mathbb K}}/{\mathbb K}_{1})$.
If $\tau \in \Gamma_{1}$, then (as the identity $I:C \to C=C^{\tau}$ conjugates $\langle \varphi \rangle=\langle \varphi \rangle^{\tau}=\langle \varphi^{\tau}\rangle$ to itself), there is a (unique) automorphism $g_{\tau}$ of 
${\mathbb P}_{\overline{{\mathbb K}}}^{1}$ such that $\pi^{\tau}=\pi^{\tau} \circ I=g_{\tau} \circ \pi$ (see the following diagram).

$$\begin{CD}
C @>I>> C=C^{\tau}\\ 
@V{\pi}VV @V{\pi^{\tau}}VV\\ 
{\mathbb P}_{\overline{{\mathbb K}}}^{1} @>g_{\tau}>>{\mathbb P}_{\overline{{\mathbb K}}}^{1} 
\end{CD}
$$

\medskip

As the group of automorphisms of ${\mathbb P}_{\overline{{\mathbb K}}}^{1}$ is given by M\"obius transformations (i.e., elements of ${\rm PGL}_{2}(\overline{{\mathbb K}})$), we must have 
$g_{\tau} \in {\rm PGL}_{2}(\overline{{\mathbb K}})$.

We may apply each $\sigma \in \Gamma_{1}$ to the above diagram to obtain the following one

$$\begin{CD}
C^{\sigma}=C @>I>> C=C^{\sigma\tau}\\ 
@V{\pi^{\sigma}}VV @V{\pi^{\sigma\tau}}VV\\ 
{\mathbb P}_{\overline{{\mathbb K}}}^{1} @>g_{\tau}^{\sigma}>>{\mathbb P}_{\overline{{\mathbb K}}}^{1} 
\end{CD}
$$

The above permits to obtain the following diagram

$$\begin{CD}
C @>I>> C=C^{\sigma}@>I=I^{\sigma}>>C=C^{\sigma \tau}\\ 
@V{\pi}VV @V{\pi^{\sigma}}VV @V{\pi^{\sigma \tau}}VV\\ {\mathbb P}_{\overline{{\mathbb K}}}^{1} @>g_{\sigma}>>{\mathbb P}_{\overline{{\mathbb K}}}^{1}@>g_{\tau}^{\sigma}>> {\mathbb P}_{\overline{{\mathbb K}}}^{1} 
\end{CD}
$$

\medskip

As the transformation $g_{\rho}$ is uniquely determined by $\rho \in \Gamma_{1}$, 
the collection $\{g_{\rho}\}_{\rho \in \Gamma_{1}}$ satisfies the co-cycle relation 
$$g_{\sigma \tau}=g_{\tau}^{\sigma} \circ g_{\sigma}, \quad \sigma, \tau \in \Gamma_{1}.$$

\medskip

Weil's descent theorem \cite{Weil} ensures the existence of a genus zero irreducible and non-singular algebraic curve $B$, defined over ${\mathbb K}_{1}$, and an isomorphism $R:{\mathbb P}_{\overline{{\mathbb K}}}^{1} \to B$, defined over $\overline{{\mathbb K}}$, so that $$g_{\sigma} \circ R^{\sigma} = R, \quad \sigma \in \Gamma_{1}.$$

Also, as for $\sigma \in \Gamma_{1}$, we have $\{\sigma(a_{1}),..., \sigma(a_{m})\}=\{g_{\sigma}(a_{1}),..., g_{\sigma}(a_{m})\}$, it follows that $\{R(a_{1}),...,R(a_{m})\}$ is $\Gamma_{1}$-invariant. 

\medskip

Let us denote by $A(n_{j})$ the set of those $a_{k}$'s for which $n_{k}=n_{j}$. 

\begin{claim}\label{claim2}
Each set $R(A(n_{j}))$ is $\Gamma_{1}$-invariant. 
\end{claim}
\begin{proof}
If $\sigma \in \Gamma_{1}$, then (as $\pi^{\sigma}=g_{\sigma} \circ \pi$) the set $g_{\sigma}(A(n_{j}))$ corresponds to the set of those $\sigma(a_{k})$ having the same $n_{l}$ (for some $l$), that is, $g_{\sigma}(A(n_{j}))=\sigma(A(n_{l}))$.
As $\varphi^{\sigma}=\varphi$, we must have $n_{l}=n_{j}$, that is, $g_{\sigma}(A(n_{j}))=\sigma(A(n_{j}))$. This last equality implies the desired claim.
\end{proof}

\begin{claim}\label{claim3}
There is an effective ${\mathbb K}_{1}$-rational divisor $U \geq 0$ of degree at most two in $B$.
\end{claim}
\begin{proof}
We follows similar techniques as used by Huggings in her Thesis \cite{Huggins} (and other authors). Let us consider 
any ${\mathbb K}_{1}$-rational meromorphic $1$-form $\omega$ in $B$. Since $B$ has genus zero, the canonical divisor $K=(\omega)$ is a ${\mathbb K}_{1}$-rational of degree $-2$. In this way, there is a positive integer $d$ such that the divisor $D=R(a_{1})+ \cdots +R(a_{m})+d K$ is ${\mathbb K}_{1}$-rational of degree $1$ or $2$. If $D \geq 0$, then we set $U:=D$. 

Let us assume $D$ is not effective. Let us consider the Riemann-Roch space $L(D)$, consisting of those non-constant rational maps $\phi:B \to {\mathbb P}_{\overline{{\mathbb K}}}^{1}$ whose divisors satisfy $(\phi)+D \geq 0$ together the constant ones.
As the divisor $D$ is ${\mathbb K}_{1}$-rational, for every $\sigma \in \Gamma_{1}$ and every $\phi \in L(D)$, it follows that $\phi^{\sigma} \in L(D)$. This, in particular, permits to observe that we can find a basis of $L(D)$ consisting of rational maps defined over ${\mathbb K}_{1}$. One of the elements of such a basis must be a non-zero constant map.
As, by Riemann-Roch's theorem, $L(D)$ has dimension $2$ (if $D$ has degree one) or $3$ (if $D$ has degree two), we may find a non-constant $f \in L(D)$ belonging to  such a basis (defined over ${\mathbb K}_{1}$). In this case, we may take  $U=(f)+D \geq 0$. 
\end{proof}

By Claim \ref{claim3}, there is an effective ${\mathbb K}_{1}$-rational divisor $U$ of degree $1$ or $2$ and $U \geq 0$. We have three possibilities:
\begin{enumerate}
\item $U=s$, where $s \in B$ is ${\mathbb K}_{1}$-rational; or
\item $U=2t$, where $t \in B$ is ${\mathbb K}_{1}$-rational; or
\item $U=r+q$, where $r,q \in B$, $r \neq q$, and $\{r,q\}$ is $\Gamma_{1}$-invariant.
\end{enumerate}

In cases (1) and (2) we have the existence of a ${\mathbb K}_{1}$-rational point in $B$. In this case, we set ${\mathbb K}_{2}={\mathbb K}_{1}$.
In case (3) we have a point (say $r$) in $B$ which is rational over a quadratic extension ${\mathbb K}_{2}$ of ${\mathbb K}_{1}$. 

\medskip

Let $b \in B$ be a  ${\mathbb K}_{2}$-rational point (whose existence is provided above). By Riemann-Roch's theorem, the Riemann-Roch space $L(b)$ (where $b$ is though as a dividor of degree one) has dimension $2$. Similarly as above, we may chose a basis $\{1,L\}$ of $L(b)$, each element defined over ${\mathbb K}_{2}$. In this case, $L:B \to \widehat{\mathbb C}$ turns out to be an isomorphism defined over ${\mathbb K}_{2}$. 

We have that $Q=L \circ R \circ \pi:C \to \widehat{\mathbb C}$ is a Galois (branched) covering with deck group $\langle \varphi \rangle$ and whose branch values are 
$\{L(R(a_{1})),...,L(R(a_{m}))\}$. It follows that  $S$ is $p$-gonally defined by
$$y^{p}=F(x)=\prod_{j=1}^{m}\left(x-L(R(a_{j}))\right)^{n_{j}}.$$

As the sets $\{L(R(a_{1})),...,L(R(a_{m}))\}$ and $L(R(A(n_{j})))$ are ${\rm Gal}(\overline{{\mathbb K}}/{\mathbb K}_{2})$-invariant (by Claim \ref{claim2} and the fact that ${\mathbb K}_{1}$ is a subfield of ${\mathbb K}_{2}$), it follows that $F(x)=\prod_{j=1}^{m}\left(x-L(R(a_{j}))\right)^{n_{j}}\in {\mathbb K}_{2}[x]$. As ${\mathbb K}_{2}$ is an extension of degree at most two of ${\mathbb K}_{1}$ and the last one is an extension of degree at most $p-1$ of ${\mathbb K}$, we are done.

\subsubsection{Proof of Parts (2) and (3)}
If $\varphi$ is already defined over ${\mathbb K}$ then we assume ${\mathbb K}_{1}={\mathbb K}$ (i.e., we set $\Gamma_{1}=\Gamma$) in the above arguments. Similarly, if in Equation \eqref{eq1} we have that $n_{1}=\cdots=n_{m}=n$, then there will be only one set $A(n)$. In this case, in the previous arguments we do not need to use Claim \ref{claim2} (where it was needed the choice of ${\mathbb K}_{1}$) and we may work as in the proof of Part (1) with ${\mathbb K}$ instead of ${\mathbb K}_{1}$.

 \subsection*{Acknowledgements} The author would like to thank both referees for their valuable com- ments, suggestions and corrections which helped to improve the paper.


\end{document}